\documentclass{amsart}
\usepackage{amssymb,amsmath,amsthm}
\usepackage[dvips]{graphicx}

\let\oldmarginpar\marginpar
\renewcommand\marginpar[1]{\-\oldmarginpar[\raggedleft\footnotesize #1]%
{\raggedright\footnotesize #1}}

\begin{document}

\newtheorem{theorem}{Theorem}[section]
\newtheorem{corollary}[theorem]{Corollary}
\newtheorem{lemma}[theorem]{Lemma}
\newtheorem{proposition}[theorem]{Proposition}
\theoremstyle{definition}
\newtheorem{definition}[theorem]{Definition}
\theoremstyle{remark}
\newtheorem{remark}[theorem]{Remark}
\theoremstyle{definition}
\newtheorem{example}[theorem]{Example}

\numberwithin{equation}{section}

\def\R{{\mathbb R}}
\def\H{{\mathbb H}}
\def\rank{{\text{rank}\,}}
\def\bd{{\partial}}

\title[H-anti-invariant submersions]{H-anti-invariant submersions from almost quaternionic Hermitian manifolds}

\author{Kwang-Soon Park}
\address{Department of Mathematical Sciences, Seoul National University, Seoul 151-747, Republic of Korea}
\email{parkksn@gmail.com}
%\thanks{The research of the first author was supported by the National Research Foundation of Korea (NRF)
%grant funded by the Korea government (MEST)(No. 2009-0057445).}

\keywords{Riemannian submersion; Lagrangian; totally geodesic; decomposition theorem}

\subjclass[2000]{53C15; 53C26.}   %Primary 57M50; Secondary 57N16, 53A20, 53C15}

\begin{abstract}
As a generalization of anti-invariant Riemannian submersions and Lagrangian Riemannian submersions,
we introduce the notions of h-anti-invariant submersions and h-Lagrangian submersions from almost quaternionic
Hermitian manifolds onto Riemannian manifolds. We obtain characterizations and investigate some properties: the integrability of distributions,
the geometry of foliations, and the harmonicity of such maps. We also find a condition for such maps to be totally geodesic and
give some examples of such maps. Finally, we obtain some types of decomposition theorems.
\end{abstract}

\maketitle
\section{Introduction}\label{intro}
\addcontentsline{toc}{section}{Introduction}

In 1960s, B. O'Neill \cite{O} and A. Gray \cite{G} introduced independently the notion of a Riemannian submersion, which is useful in many areas: 
physics (\cite{BL2}, \cite{W2}, \cite{BL}, \cite{IV}, \cite{IV2}, \cite{M}),  medical imaging \cite{MST}, robotic theory \cite{A} (see \cite{S1}).

In 1976, B. Watson \cite{W} defined almost Hermitian submersions, which are Riemannian submersions from almost Hermitian manifolds onto
almost Hermitian manifolds. Using this notion, he investigates a kind of structural problems among base manifold, fibers, total manifold.
This notion was extended to almost contact manifolds \cite{C2}, locally conformal K\"{a}hler manifolds \cite{MR}, and
quaternion K\"{a}hler manifolds \cite{IMV}.

In 2010, B. Sahin \cite{S0} introduced the notions of anti-invariant Riemannian submersions and Lagrangian Riemannian submersions
 from almost Hermitian manifolds onto Riemannian manifolds.
Using this notions, he studies total manifolds. In particular, he investigates some kinds of decomposition theorems.

We know that Riemannian submersions are related with physics and have
 applications in Yang-Mills theory (\cite{BL2},
\cite{W2}), Kaluza-Klein theory (\cite{BL}, \cite{IV}),
Supergravity and superstring theories (\cite{IV2}, \cite{M}).
And quaternionic K\"{a}hler manifolds have applications in physics as the target
spaces for nonlinear $\sigma-$models with supersymmetry \cite{CMMS}.

The paper is organized as follows. In section 2 we recall some notions, which are needed in the later sections.
In section 3 we introduce the notions of h-anti-invariant submersions and h-Lagrangian submersions from almost quaternionic Hermitian
manifolds onto Riemannian manifolds, give examples, and investigate some properties:
the integrability of distributions, the geometry of foliations, the condition for such maps to be totally geodesic,
and the condition for such maps to be harmonic.
In section 4 under h-anti-invariant submersions and h-Lagrangian submersions, we consider some decomposition theorems.

\section{Preliminaries}\label{Prel}

Let $(M,g,J)$ be an almost Hermitian manifold, where $M$ is a $C^{\infty}$-manifold, $g$ is a Riemannian metric on $M$, and $J$ is
a compatible almost complex structure on $(M,g)$. (i.e., $J\in End(TM)$, $J^2 = -id$, $g(JX, JY) = g(X, Y)$ for $X,Y\in \Gamma(TM)$.)

We call $(M,g,J)$ a {\em K\"{a}hler manifold} if $\nabla J = 0$, where $\nabla$ is the Levi-Civita connection of $g$.

Let $(M,g_M)$ and $(N,g_N)$ be Riemannian manifolds.

Let $F : (M,g_M) \mapsto (N,g_N)$ be a $C^{\infty}$-map.

The {\em second fundamental form} of $F$ is given by
$$
(\nabla F_*)(U,V) := \nabla^F _U F_* V-F_* (\nabla _U V) \quad
\text{for} \ U,V\in \Gamma(TM),
$$
where $\nabla^F$ is the pullback connection along $F$ and $\nabla$ is the
Levi-Civita connection of $g_M$ \cite{BW}.

Then the map $F$ is {\em harmonic} if and only if $trace (\nabla F_*) = 0$ \cite{BW}.

We call $F$ a {\em totally geodesic map} if $(\nabla F_*)(U,V) = 0$ for
$U,V\in \Gamma (TM)$ \cite{BW}.

The map $F$ is said to be a {\em $C^{\infty}$-submersion} if $F$ is surjective and the differential $(F_*)_p$  has a maximal rank for any $p\in M$.

We call $F$ a {\em Riemannian submersion} (\cite{O}, \cite{FIP}) if $F$ is a $C^{\infty}$-submersion and
\begin{equation}\label{eq: 2-1}
(F_*)_p : ((\ker (F_*)_p)^{\perp}, (g_M)_p) \mapsto (T_{F(p)} N, (g_N)_{F(p)})
\end{equation}
is a linear isometry for any $p\in M$,
where $(\ker (F_*)_p)^{\perp}$ is the orthogonal complement of the space $\ker (F_*)_p$ in the tangent space $T_p M$ to $M$ at $p$.

Let $F : (M,g_M) \mapsto (N,g_N)$ be a Riemannian submersion.

For any vector field $U\in \Gamma(TM)$, we write
\begin{equation}\label{eq: 2-2}
U = \mathcal{V}U + \mathcal{H}U,
\end{equation}
where $\mathcal{V}U\in \Gamma(\ker F_*)$ and $\mathcal{H}U\in
\Gamma((\ker F_*)^{\perp})$.

Define the O'Neill tensors $\mathcal{T}$ and $\mathcal{A}$ by
\begin{eqnarray}
  \mathcal{A}_U V &=& \mathcal{H}\nabla_{\mathcal{H}U} \mathcal{V}V + \mathcal{V}\nabla_{\mathcal{H}U} \mathcal{H}V,  \label{eq: 2-3} \\
  \mathcal{T}_U V &=& \mathcal{H}\nabla_{\mathcal{V}U} \mathcal{V}V + \mathcal{V}\nabla_{\mathcal{V}U} \mathcal{H}V  \label{eq: 2-4}
\end{eqnarray}
for $U, V\in \Gamma(TM)$, where $\nabla$ is the
Levi-Civita connection of $g_M$ (\cite{O}, \cite{FIP}).

Let
\begin{equation}\label{eq: 2-5}
\widehat{\nabla}_V W := \mathcal{V}\nabla_V W \quad \text{for} \ V,W\in \Gamma(\ker F_*).
\end{equation}
Then we have
\begin{eqnarray}
  \mathcal{A}_X Y &=& -\mathcal{A}_Y X = \frac{1}{2}\mathcal{V}[X, Y],   \label{eq: 2-6} \\
  \mathcal{T}_U V &=& \mathcal{T}_V U  \label{eq: 2-7}
\end{eqnarray}
for $X,Y\in \Gamma((\ker F_*)^{\perp})$ and $U,V\in \Gamma(\ker F_*)$.

\begin{proposition}\label{prop: 2-1}  (\cite{O}, \cite{FIP})
Let $F$ be a Riemannian submersion from a Riemannian manifold $(M,g_M)$ onto a Riemannian manifold $(N,g_N)$.
Then we obtain
\begin{eqnarray}
  g_M(\mathcal{T}_U V, W) &=& -g_M(V, \mathcal{T}_U W), \label{eq: 2-71} \\
  g_M(\mathcal{A}_U V, W) &=& -g_M(V, \mathcal{A}_U W), \label{eq: 2-72} \\
  (\nabla F_*)(U,V) &=& (\nabla F_*)(V,U),   \label{eq: 2-8} \\
  (\nabla F_*)(X,Y) &=& 0  \label{eq: 2-9}
  \end{eqnarray}
for $U,V,W\in \Gamma(TM)$ and $X,Y\in \Gamma((\ker F_*)^{\perp})$.
\end{proposition}

We remind the notions of an anti-invariant Riemannian submersion and a Lagrangian Riemannian submersion.

Let $F$ be a Riemannian submersion from an almost Hermitian manifold $(M,g_M,J)$ onto a Riemannian manifold $(N,g_N)$.
The map $F$ is said to be an {\em anti-invariant Riemannian submersion} \cite{S0} if $J(\ker F_*) \subset (\ker F_*)^{\perp}$.

We call $F$ a {\em Lagrangian Riemannian submersion} \cite{S0} if $J(\ker F_*) = (\ker F_*)^{\perp}$.

Let $M$ be a $4m-$dimensional $C^{\infty}$-manifold and let $E$ be a rank 3 subbundle of
$End (TM)$ such that for any point $p\in M$ with a
neighborhood $U$, there exists a local basis $\{ J_1,J_2,J_3 \}$
of sections of $E$ on $U$ satisfying for all $\alpha\in \{ 1,2,3 \}$
$$
J_{\alpha}^2=-id, \quad
J_{\alpha}J_{\alpha+1}=-J_{\alpha+1}J_{\alpha}=J_{\alpha+2},
$$
where the indices are taken from $\{ 1,2,3 \}$ modulo 3.

Then we call $E$ an {\em almost quaternionic structure} on $M$ and $(M,E)$ an {\em almost quaternionic manifold} \cite{AM}.

Moreover, let $g$ be a Riemannian metric on $M$ such that for any point $p\in M$ with a
neighborhood $U$, there exists a local basis $\{ J_1,J_2,J_3 \}$
of sections of $E$ on $U$ satisfying for all $\alpha\in \{ 1,2,3 \}$
\begin{equation}\label{hypercom}
J_{\alpha}^2=-id, \quad
J_{\alpha}J_{\alpha+1}=-J_{\alpha+1}J_{\alpha}=J_{\alpha+2},
\end{equation}
\begin{equation}\label{hypermet}
g(J_{\alpha}X, J_{\alpha}Y)=g(X, Y)
\end{equation}
for  $X, Y\in \Gamma(TM)$, where the indices are
taken from $\{ 1,2,3 \}$ modulo 3.

Then we call $(M,E,g)$ an {\em almost quaternionic Hermitian manifold} \cite{IMV}.

For convenience, the above basis $\{ J_1,J_2,J_3 \}$ satisfying (\ref{hypercom}) and (\ref{hypermet}) is said to be a {\em quaternionic Hermitian basis}.

Let $(M,E,g)$ be an almost quaternionic Hermitian manifold.

We call $(M,E,g)$ a {\em quaternionic K\"{a}hler manifold} if given a point $p\in M$ with a neighborhood $U$, there exist
 1-forms $\omega_1, \omega_2, \omega_3$ on $U$ such that
for any $\alpha \in \{ 1,2,3 \}$,
$$
\nabla_X J_{\alpha} =
\omega_{\alpha+2}(X)J_{\alpha+1}-\omega_{\alpha+1}(X)J_{\alpha+2}
$$
for $X\in \Gamma(TM)$, where the indices are
taken from $\{ 1,2,3 \}$ modulo 3 \cite{IMV}.

If there exists a global parallel quaternionic Hermitian basis $\{ J_1,J_2,J_3 \}$ of
sections of $E$ on $M$ (i.e., $\nabla J_{\alpha} = 0$ for $\alpha \in \{ 1,2,3 \}$, where $\nabla$ is the Levi-Civita connection of $g$),
then $(M, E, g)$ is said to be a {\em hyperk\"{a}hler manifold}.
Furthermore, we call $(J_1, J_2, J_3, g )$ a {\em hyperk\"{a}hler structure} on $M$ and $g$ a {\em hyperk\"{a}hler metric} \cite{B}.

Now, we recall the notions of almost h-slant submersions, almost h-semi-invariant submersions, and almost h-semi-slant submersions.

Let $(M, E, g_M)$ be an almost quaternionic Hermitian manifold and $(N, g_N)$ a Riemannian manifold.

A Riemannian submersion $F : (M, E, g_M) \mapsto (N,
g_N)$ is said to be an {\em almost h-slant submersion} if given a
point $p\in M$ with a neighborhood $U$, there exists a
 quaternionic Hermitian basis $\{ I,J,K \}$ of sections of
$E$ on $U$ such that for $R\in \{ I,J,K \}$ the angle
$\theta_R(X)$ between $RX$ and the space $\ker (F_*)_q$ is
constant for nonzero $X\in \ker (F_*)_q$ and $q\in U$ \cite{P}.

A Riemannian submersion $F : (M,E,g_M) \mapsto (N,g_N)$ is called an {\em almost
h-semi-invariant submersion} if given a point $p\in M$ with a neighborhood $U$, there exists a  quaternionic Hermitian
basis $\{ I,J,K \}$ of sections of $E$ on $U$ such that for each
$R\in \{ I,J,K \}$, there is a distribution $\mathcal{D}_1^R
\subset \ker F_*$ on $U$ such that
$$
\ker F_* =\mathcal{D}_1^R\oplus \mathcal{D}_2^R, \
R(\mathcal{D}_1^R)=\mathcal{D}_1^R, \ R(\mathcal{D}_2^R)\subset
(\ker F_*)^{\perp},
$$
where $\mathcal{D}_2^R$ is the orthogonal complement of
$\mathcal{D}_1^R$ in $\ker F_*$ \cite{P2}.

A Riemannian submersion $F : (M,E,g_M) \mapsto (N,g_N)$ is called an {\em almost h-semi-slant submersions} if
given a point $p\in M$ with a neighborhood
$U$, there exists a  quaternionic Hermitian basis $\{ I,J,K \}$ of
sections of $E$ on $U$ such that for each $R\in \{ I,J,K \}$,
there is a distribution $\mathcal{D}_1^R \subset \ker F_*$ on $U$
such that
$$
\ker F_* =\mathcal{D}_1^R\oplus \mathcal{D}_2^R, \
R(\mathcal{D}_1^R)=\mathcal{D}_1^R,
$$
and the angle $\theta_R=\theta_R(X)$ between $RX$ and the space
$(\mathcal{D}_2^R)_q$ is constant for nonzero $X\in
(\mathcal{D}_2^R)_q$ and $q\in U$, where $\mathcal{D}_2^R$ is the
orthogonal complement of $\mathcal{D}_1^R$ in $\ker F_*$ \cite{P3}.

Throughout this paper, we will use the above notations.

\section{H-anti-invariant submersions}\label{semi}

In this section, we introduce the notions of h-anti-invariant submersions and h-Lagrangian submersions from almost quaternionic Hermitian manifolds
onto Riemannian manifolds and investigate their properties.

\begin{definition}
Let $(M,E,g_M)$ be an almost quaternionic Hermitian manifold and
$(N,g_N)$ a Riemannian manifold. Let $F : (M,E,g_M)\mapsto (N,g_N)$ be a Riemannian submersion.
We call the map $F$ a {\em h-anti-invariant submersion} if given a point $p\in M$ with a neighborhood
$U$, there exists a  quaternionic Hermitian basis $\{ I,J,K \}$ of
sections of $E$ on $U$ such that $R(\ker F_*) \subset (\ker F_*)^{\perp}$ for $R\in \{ I,J,K \}$.
\end{definition}

We call such a basis $\{ I,J,K \}$ a {\em h-anti-invariant basis}.

\begin{remark}
As we see, a h-anti-invariant submersion is one of the particular cases of
an almost h-slant submersion, an almost h-semi-invariant submersion, and an almost h-semi-slant submersion.
\end{remark}

\begin{remark}\label{rem: 3-1}
Let $F$ be a h-anti-invariant submersion from an almost quaternionic Hermitian manifold $(M,E,g_M)$ onto
a Riemannian manifold $(N,g_N)$. Then there does not exist such a map $F$ such that $\dim (\ker F_*) = \dim ((\ker F_*)^{\perp})$.
If not, then given a local quaternionic Hermitian basis $\{ I,J,K \}$ of $E$ with $R(\ker F_*) \subset (\ker F_*)^{\perp}$
for $R\in \{ I,J,K \}$, we have
$$
R(\ker F_*) = (\ker F_*)^{\perp}  \quad \text{for} \ R\in \{ I,J,K \}
$$
so that
$$
K(\ker F_*) = IJ(\ker F_*) = I((\ker F_*)^{\perp}) = (\ker F_*),
$$
contradiction!
\end{remark}

From Remark \ref{rem: 3-1}, we need to define another type of such a map.

\begin{definition}
Let $(M,E,g_M)$ be an almost quaternionic Hermitian manifold and
$(N,g_N)$ a Riemannian manifold. Let $F : (M,E,g_M)\mapsto (N,g_N)$ be a Riemannian submersion.
We call the map $F$ a {\em h-Lagrangian submersion} if given a point $p\in M$ with a neighborhood
$U$, there exists a  quaternionic Hermitian basis $\{ I,J,K \}$ of
sections of $E$ on $U$ such that $I(\ker F_*) = (\ker F_*)^{\perp}$, $J(\ker F_*) = \ker F_*$,
and $K(\ker F_*) = (\ker F_*)^{\perp}$.
\end{definition}

We call such a basis $\{ I,J,K \}$ a {\em h-Lagrangian basis}.

\begin{remark}
(a) It is easy to check that $J(\ker F_*) = \ker F_*$ implies $J((\ker F_*)^{\perp}) = (\ker F_*)^{\perp}$.

(b) Let $F$ be a Riemannian submersion from an almost quaternionic Hermitian manifold $(M,E,g_M)$ onto
a Riemannian manifold $(N,g_N)$ such that $\dim (\ker F_*) = \dim ((\ker F_*)^{\perp})$. Then there does not exist
such a map $F$ such that for some local quaternionic Hermitian basis $\{ I,J,K \}$ of $E$, we have
$$
I(\ker F_*) = \ker F_*, \ J(\ker F_*) = \ker F_*, \ K(\ker F_*) = (\ker F_*)^{\perp}.
$$
If not, then $K(\ker F_*) = IJ(\ker F_*) = I(\ker F_*) = \ker F_*$, contradiction!
\end{remark}

Now, we give some examples.
Note that given an Euclidean space $\mathbb{R}^{4m}$ with
coordinates $(x_1,x_2,\cdots,x_{4m})$, we can canonically choose
complex structures $I, J, K$ on $\mathbb{R}^{4m}$ as follows:
\begin{align*}
  &I(\tfrac{\partial}{\partial x_{4k+1}})=\tfrac{\partial}{\partial x_{4k+2}},
  I(\tfrac{\partial}{\partial x_{4k+2}})=-\tfrac{\partial}{\partial x_{4k+1}},
  I(\tfrac{\partial}{\partial x_{4k+3}})=\tfrac{\partial}{\partial x_{4k+4}},
  I(\tfrac{\partial}{\partial x_{4k+4}})=-\tfrac{\partial}{\partial x_{4k+3}},     \\
  &J(\tfrac{\partial}{\partial x_{4k+1}})=\tfrac{\partial}{\partial x_{4k+3}},
  J(\tfrac{\partial}{\partial x_{4k+2}})=-\tfrac{\partial}{\partial x_{4k+4}},
  J(\tfrac{\partial}{\partial x_{4k+3}})=-\tfrac{\partial}{\partial x_{4k+1}},
  J(\tfrac{\partial}{\partial x_{4k+4}})=\tfrac{\partial}{\partial x_{4k+2}},    \\
  &K(\tfrac{\partial}{\partial x_{4k+1}})=\tfrac{\partial}{\partial x_{4k+4}},
  K(\tfrac{\partial}{\partial x_{4k+2}})=\tfrac{\partial}{\partial x_{4k+3}},
  K(\tfrac{\partial}{\partial x_{4k+3}})=-\tfrac{\partial}{\partial x_{4k+2}},
  K(\tfrac{\partial}{\partial x_{4k+4}})=-\tfrac{\partial}{\partial x_{4k+1}}
\end{align*}
for $k\in \{ 0,1,\cdots,m-1 \}$.

Then we easily check that $(I,J,K,\langle \ ,\ \rangle)$ is a hyperk\"{a}hler structure on $\mathbb{R}^{4m}$,
where $\langle \ ,\ \rangle$ denotes the Euclidean metric on $\mathbb{R}^{4m}$.

\begin{example}
Define a map $F : \mathbb{R}^{12} \mapsto \mathbb{R}^9$ by
$$
F(x_1,\cdots,x_{12}) = (x_{10},x_{11},x_{12},x_4,x_3,x_2,x_8,x_6,x_7).
$$
Then the map $F$ is a h-anti-invariant submersion such that
\begin{align*}
&\ker F_* = < \frac{\partial}{\partial x_1}, \frac{\partial}{\partial x_5},\frac{\partial}{\partial x_9} >,     \\
&(\ker F_*)^{\perp} = < \frac{\partial}{\partial x_2}, \frac{\partial}{\partial x_3},\frac{\partial}{\partial x_4},
  \frac{\partial}{\partial x_6}, \frac{\partial}{\partial x_7},\frac{\partial}{\partial x_8},
  \frac{\partial}{\partial x_{10}}, \frac{\partial}{\partial x_{11}},\frac{\partial}{\partial x_{12}}>,     \\
&I(\frac{\partial}{\partial x_1}) = \frac{\partial}{\partial x_2},  I(\frac{\partial}{\partial x_5}) = \frac{\partial}{\partial x_6},
I(\frac{\partial}{\partial x_9}) = \frac{\partial}{\partial x_{10}},   \\
&J(\frac{\partial}{\partial x_1}) = \frac{\partial}{\partial x_3},  J(\frac{\partial}{\partial x_5}) = \frac{\partial}{\partial x_7},
J(\frac{\partial}{\partial x_9}) = \frac{\partial}{\partial x_{11}},   \\
&K(\frac{\partial}{\partial x_1}) = \frac{\partial}{\partial x_4}, K(\frac{\partial}{\partial x_5}) = \frac{\partial}{\partial x_8},
K(\frac{\partial}{\partial x_9}) = \frac{\partial}{\partial x_{12}}.
\end{align*}
\end{example}

\begin{example}
Define a map $F : \mathbb{R}^{4} \mapsto \mathbb{R}^2$ by
$$
F(x_1,\cdots,x_{4}) = (\frac{x_{2} + x_{3}}{\sqrt{2}},\frac{x_{1} + x_{4}}{\sqrt{2}}).
$$
Then the map $F$ is a h-Lagrangian submersion such that
\begin{align*}
&\ker F_* = < V_1 = \frac{\partial}{\partial x_2} - \frac{\partial}{\partial x_3},
V_2 = \frac{\partial}{\partial x_1} - \frac{\partial}{\partial x_4} >,     \\
&(\ker F_*)^{\perp} = < X_1 = \frac{\partial}{\partial x_2} + \frac{\partial}{\partial x_3},
  X_2 = \frac{\partial}{\partial x_1} + \frac{\partial}{\partial x_4}>,     \\
&I(V_1) = -X_2,  I(V_2) = X_1,   \\
&J(V_1) = V_2, J(V_2) = -V_1,   \\
&K(V_1) = X_1, K(V_2) = X_2.
\end{align*}
\end{example}

Let $F$ be a h-anti-invariant submersion (or a h-Lagrangian submersion, respectively) from an almost quaternionic Hermitian manifold $(M,E,g_M)$
onto a Riemannian manifold $(N,g_N)$. Given a point $p\in M$ with a neighborhood $U$, we have a h-anti-invariant basis
(or a h-Lagrangian basis, respectively) $\{ I,J,K \}$ of sections of $E$ on $U$.

Then given $X\in \Gamma((\ker F_*)^{\perp})$ and $R\in \{ I,J,K \}$, we write
\begin{equation}\label{eq: 3-1}
RX = B_R X + C_R X,
\end{equation}
where $B_R X\in \Gamma(\ker F_*)$ and $C_R X\in \Gamma((\ker F_*)^{\perp})$.

If $F : (M,E,g_M) \mapsto (N,g_N)$ is a h-anti-invariant submersion, then we get
\begin{equation}\label{eq: 3-2}
(\ker F_*)^{\perp} = R(\ker F_*) \oplus \mu_R \quad \text{for} \ R\in \{ I,J,K \}.
\end{equation}
Then it is easy to check that $\mu_R$ is $R$-invariant for $R\in \{ I,J,K \}$.

Given $X\in \Gamma((\ker F_*)^{\perp})$ and $R\in \{ I,J,K \}$, we have
\begin{equation}\label{eq: 3-3}
X = P_R X + Q_R X,
\end{equation}
where $P_R X\in \Gamma(R(\ker F_*))$ and $Q_R X\in \Gamma(\mu_R)$.

Furthermore, given $R\in \{ I,J,K \}$, we obtain
\begin{equation}\label{eq: 3-4}
C_R X\in \Gamma(\mu_R)   \quad \text{for} \ X\in \Gamma((\ker F_*)^{\perp}) 
\end{equation}
and
\begin{equation}\label{eq: 3-5}
g_M(C_R X, RV) = 0  \quad \text{for} \ V\in \Gamma(\ker F_*).
\end{equation}
Then it is easy to have

\begin{lemma}
Let $F$ be a h-anti-invariant submersion from a
hyperk\"{a}hler manifold $(M,I,J,K,g_M)$ onto a Riemannian
manifold $(N, g_N)$ such that $(I,J,K)$ is a h-anti-invariant
basis. Then we get

\begin{enumerate}
\item
\begin{align*}
  &\mathcal{T}_V RW = B_R \mathcal{T}_V W    \\
  &\mathcal{H}\nabla_V RW = C_R \mathcal{T}_V W + R \widehat{\nabla}_V W
\end{align*}
for $V,W\in \Gamma(\ker F_*)$ and $R\in \{ I,J,K \}$.
\item
\begin{align*}
  &\mathcal{A}_X C_R Y + \mathcal{V}\nabla_X B_R Y = B_R \mathcal{H} \nabla_X Y    \\
  &\mathcal{H} \nabla_X C_R Y + \mathcal{A}_X B_R Y = R\mathcal{A}_X Y + C_R \mathcal{H}\nabla_X Y
\end{align*}
for $X,Y\in \Gamma((\ker F_*)^{\perp})$ and $R\in \{ I,J,K \}$.
\item
\begin{align*}
  &\mathcal{A}_X RV = B_R \mathcal{A}_X V   \\
  &\mathcal{H}\nabla_X RV = C_R \mathcal{A}_X V + R\mathcal{V}\nabla_X V
\end{align*}
for $V\in \Gamma(\ker F_*)$, $X\in \Gamma((\ker F_*)^{\perp})$,
and $R\in \{ I,J,K \}$.
\end{enumerate}
\end{lemma}

\begin{theorem}\label{thm: 3-1}
Let $F$ be a h-anti-invariant submersion from a hyperk\"{a}hler
manifold $(M,I,J,K,g_M)$ onto a Riemannian manifold $(N, g_N)$
such that $(I,J,K)$ is a h-anti-invariant basis. Then the following
conditions are equivalent:

a) the distribution $(\ker F_*)^{\perp}$ is integrable.

b)
$$
g_M(\mathcal{A}_X B_I Y - \mathcal{A}_Y B_I X, IV) = g_M(C_I Y, I\mathcal{A}_X V) - g_M(C_I X, I\mathcal{A}_Y V)
$$
for $V\in \Gamma(\ker F_*)$ and $X,Y\in \Gamma((\ker F_*)^{\perp})$.

c)
$$
g_M(\mathcal{A}_X B_J Y - \mathcal{A}_Y B_J X, JV) = g_M(C_J Y, J\mathcal{A}_X V) - g_M(C_J X, J\mathcal{A}_Y V)
$$
for $V\in \Gamma(\ker F_*)$ and $X,Y\in \Gamma((\ker F_*)^{\perp})$.

d)
$$
g_M(\mathcal{A}_X B_K Y - \mathcal{A}_Y B_K X, KV) = g_M(C_K Y, K\mathcal{A}_X V) - g_M(C_K X, K\mathcal{A}_Y V)
$$
for $V\in \Gamma(\ker F_*)$ and $X,Y\in \Gamma((\ker F_*)^{\perp})$.
\end{theorem}

\begin{proof}
Given $V\in \Gamma(\ker F_*)$, $X,Y\in \Gamma((\ker F_*)^{\perp})$, and $R\in \{ I,J,K \}$, by using (\ref{eq: 3-5}), we get
\begin{align*}
g_M([X,Y], V) &= g_M(\nabla_X RY - \nabla_Y RX, RV)    \\
      &= g_M(\nabla_X B_R Y + \nabla_X C_R Y - \nabla_Y B_R X - \nabla_Y C_R X, RV)    \\
      &= g_M(\mathcal{A}_X B_R Y - \mathcal{A}_Y B_R X, RV) - g_M(C_R Y, \nabla_X RV) + g_M(C_R X, \nabla_Y RV)   \\
      &= g_M(\mathcal{A}_X B_R Y - \mathcal{A}_Y B_R X, RV) - g_M(C_R Y, R\mathcal{A}_X V) + g_M(C_R X, R\mathcal{A}_Y V).
\end{align*}
Hence,
$$
a) \Leftrightarrow b), \quad a) \Leftrightarrow c), \quad a) \Leftrightarrow d).
$$
Therefore, the result follows.
\end{proof}

\begin{lemma}
Let $F$ be a h-Lagrangian submersion from a hyperk\"{a}hler
manifold $(M,I,J,K,g_M)$ onto a Riemannian manifold $(N, g_N)$
such that $(I,J,K)$ is a h-Lagrangian basis. Then the following
conditions are equivalent:

a) the distribution $(\ker F_*)^{\perp}$ is integrable.

b) $\mathcal{A}_X IY = \mathcal{A}_Y IX$ for $X,Y\in \Gamma((\ker F_*)^{\perp})$.

c) $\mathcal{A}_X KY = \mathcal{A}_Y KX$ for $X,Y\in \Gamma((\ker F_*)^{\perp})$.

d) $\mathcal{A}_X JY = \mathcal{A}_Y JX$ for $X,Y\in \Gamma((\ker F_*)^{\perp})$.
\end{lemma}

\begin{proof}
By the proof of Theorem \ref{thm: 3-1}, we get $a) \Leftrightarrow b)$ and $a) \Leftrightarrow c)$.

Given $V\in \Gamma(\ker F_*)$ and $X,Y\in \Gamma((\ker F_*)^{\perp})$, since $J(\ker F_*) = \ker F_*$, we obtain
\begin{align*}
g_M([X,Y], JV) &= -g_M(\nabla_X JY - \nabla_Y JX, V)    \\
      &= g_M(\mathcal{A}_Y JX - \mathcal{A}_X JY, V),
\end{align*}
which implies $a) \Leftrightarrow d)$.

Therefore, the result follows.
\end{proof}

We consider the equivalent conditions for distributions to be totally geodesic.

\begin{theorem}\label{thm: 3-2}
Let $F$ be a h-anti-invariant submersion from a hyperk\"{a}hler
manifold $(M,I,J,K,g_M)$ onto a Riemannian manifold $(N, g_N)$
such that $(I,J,K)$ is a h-anti-invariant basis.
Then the following conditions are equivalent:

a) the distribution $(\ker F_*)^{\perp}$ defines a totally geodesic foliation on $M$.

b)
$$
g_M(\mathcal{A}_X B_I Y, IV) = g_M(C_I Y, I\mathcal{A}_X V)
$$
for $V\in \Gamma(\ker F_*)$ and $X,Y\in \Gamma((\ker F_*)^{\perp})$.

c)
$$
g_M(\mathcal{A}_X B_J Y, JV) = g_M(C_J Y, J\mathcal{A}_X V)
$$
for $V\in \Gamma(\ker F_*)$ and $X,Y\in \Gamma((\ker F_*)^{\perp})$.

d)
$$
g_M(\mathcal{A}_X B_K Y, KV) = g_M(C_K Y, K\mathcal{A}_X V)
$$
for $V\in \Gamma(\ker F_*)$ and $X,Y\in \Gamma((\ker F_*)^{\perp})$.
\end{theorem}

\begin{proof}
Given $V\in \Gamma(\ker F_*)$, $X,Y\in \Gamma((\ker F_*)^{\perp})$, and $R\in \{ I,J,K \}$, by using (\ref{eq: 3-5}), we have
\begin{align*}
g_M(\nabla_X Y, V) &= g_M(\nabla_X B_R Y + \nabla_X C_R Y, RV)    \\
      &= g_M(\mathcal{A}_X B_R Y, RV) - g_M(C_R Y, \nabla_X RV)   \\
      &= g_M(\mathcal{A}_X B_R Y, RV) - g_M(C_R Y, R\mathcal{A}_X V),
\end{align*}
which implies $a) \Leftrightarrow b)$, $a) \Leftrightarrow c)$, $a) \Leftrightarrow d)$.

Therefore, the result follows.
\end{proof}

\begin{lemma}\label{lem: 3-2}
Let $F$ be a h-Lagrangian submersion from a hyperk\"{a}hler
manifold $(M,I,J,K,g_M)$ onto a Riemannian manifold $(N, g_N)$
such that $(I,J,K)$ is a h-Lagrangian basis. Then the following
conditions are equivalent:

a) the distribution $(\ker F_*)^{\perp}$ defines a totally geodesic foliation on $M$.

b) $\mathcal{A}_X IY = 0$ for $X,Y\in \Gamma((\ker F_*)^{\perp})$.

c) $\mathcal{A}_X KY = 0$ for $X,Y\in \Gamma((\ker F_*)^{\perp})$.

d) $\mathcal{A}_X JY = 0$ for $X,Y\in \Gamma((\ker F_*)^{\perp})$.
\end{lemma}

\begin{proof}
By the proof of Theorem \ref{thm: 3-2}, we get $a) \Leftrightarrow b)$ and $a) \Leftrightarrow c)$.

Given $V\in \Gamma(\ker F_*)$ and $X,Y\in \Gamma((\ker F_*)^{\perp})$, we obtain
\begin{align*}
g_M(\nabla_X Y, JV) &= -g_M(\nabla_X JY, V)    \\
      &= -g_M(\mathcal{A}_X JY, V),
\end{align*}
which implies $a) \Leftrightarrow d)$.

Therefore, the result follows.
\end{proof}

\begin{theorem}\label{thm: 3-3}
Let $F$ be a h-anti-invariant submersion from a hyperk\"{a}hler
manifold $(M,I,J,K,g_M)$ onto a Riemannian manifold $(N, g_N)$
such that $(I,J,K)$ is a h-anti-invariant basis.
Then the following conditions are equivalent:

a) the distribution $\ker F_*$ defines a totally geodesic foliation on $M$.

b)
$$
\mathcal{T}_V B_I X + \mathcal{A}_{C_I X} V\in \Gamma(\mu_I)
$$
for $V\in \Gamma(\ker F_*)$ and $X\in \Gamma((\ker F_*)^{\perp})$.

c)
$$
\mathcal{T}_V B_J X + \mathcal{A}_{C_J X} V\in \Gamma(\mu_J)
$$
for $V\in \Gamma(\ker F_*)$ and $X\in \Gamma((\ker F_*)^{\perp})$.

d)
$$
\mathcal{T}_V B_K X + \mathcal{A}_{C_K X} V\in \Gamma(\mu_K)
$$
for $V\in \Gamma(\ker F_*)$ and $X\in \Gamma((\ker F_*)^{\perp})$.
\end{theorem}

\begin{proof}
Given $V,W\in \Gamma(\ker F_*)$, $X\in \Gamma((\ker F_*)^{\perp})$, and $R\in \{ I,J,K \}$, by using (\ref{eq: 3-5}), we get
\begin{align*}
g_M(\nabla_V W, X) &= g_M(\nabla_V RW, RX)    \\
      &= -g_M(RW, \nabla_V B_R X + \nabla_V C_R X)   \\
      &= -g_M(RW, \mathcal{T}_V B_R X) - g_M(RW, \nabla_V C_R X).
\end{align*}
But
\begin{align*}
g_M(RW, \nabla_V C_R X)
      &= g_N(F_* RW, F_* \nabla_V C_R X) \ (\text{since} \ RW\in \Gamma((\ker F_*)^{\perp}))    \\
      &= -g_N(F_* RW, (\nabla F_*)(V, C_R X))    \\
      &= -g_N(F_* RW, (\nabla F_*)(C_R X, V)) \ (\text{by} \ (\ref{eq: 2-8}))   \\
      &= g_M(RW, \nabla_{C_R X} V)  \\
      &= g_M(RW, \mathcal{A}_{C_R X} V).
\end{align*}
Hence,
$$
g_M(\nabla_V W, X) = -g_M(RW, \mathcal{T}_V B_R X + \mathcal{A}_{C_R X} V),
$$
which implies $a) \Leftrightarrow b)$, $a) \Leftrightarrow c)$, $a) \Leftrightarrow d)$.

Therefore, we obtain the result.
\end{proof}

\begin{lemma}\label{lem: 3-3}
Let $F$ be a h-Lagrangian submersion from a hyperk\"{a}hler
manifold $(M,I,J,K,g_M)$ onto a Riemannian manifold $(N, g_N)$
such that $(I,J,K)$ is a h-Lagrangian basis. Then the following
conditions are equivalent:

a) the distribution $\ker F_*$ defines a totally geodesic foliation on $M$.

b) $\mathcal{T}_V IX = 0$ for $X\in \Gamma((\ker F_*)^{\perp})$ and  $V\in \Gamma(\ker F_*)$.

c) $\mathcal{T}_V KX = 0$ for $X\in \Gamma((\ker F_*)^{\perp})$ and  $V\in \Gamma(\ker F_*)$.

d) $\mathcal{T}_V JX = 0$ for $X\in \Gamma((\ker F_*)^{\perp})$ and  $V\in \Gamma(\ker F_*)$.
\end{lemma}

\begin{proof}
By the proof of Theorem \ref{thm: 3-3}, we have $a) \Leftrightarrow b)$ and $a) \Leftrightarrow c)$.

Given $V,W\in \Gamma(\ker F_*)$ and $X\in \Gamma((\ker F_*)^{\perp})$, we get
\begin{align*}
g_M(\nabla_V W, JX) &= -g_M(W, \nabla_V JX)    \\
      &= -g_M(W, \mathcal{T}_V JX)  \ (\text{since} \ JX\in \Gamma((\ker F_*)^{\perp})),
\end{align*}
which implies $a) \Leftrightarrow d)$.

Therefore, the result follows.
\end{proof}

Now, we consider the equivalent conditions for such maps to be either totally geodesic or harmonic.

\begin{theorem}\label{thm: 3-4}
Let $F$ be a h-anti-invariant submersion from a hyperk\"{a}hler
manifold $(M,I,J,K,g_M)$ onto a Riemannian manifold $(N, g_N)$
such that $(I,J,K)$ is a h-anti-invariant basis.
Then the following conditions are equivalent:

a) the map $F$ is a totally geodesic map.

b)
$$
\mathcal{A}_X IV = 0, \ Q_I \mathcal{H}\nabla_X IV = 0, \ \mathcal{T}_V IW = 0, \ Q_I \mathcal{H}\nabla_V IW = 0
$$
for $V,W\in \Gamma(\ker F_*)$ and $X\in \Gamma((\ker F_*)^{\perp})$.

c)
$$
\mathcal{A}_X JV = 0, \ Q_J \mathcal{H}\nabla_X JV = 0, \ \mathcal{T}_V JW = 0, \ Q_J \mathcal{H}\nabla_V JW = 0
$$
for $V,W\in \Gamma(\ker F_*)$ and $X\in \Gamma((\ker F_*)^{\perp})$.

d)
$$
\mathcal{A}_X KV = 0, \ Q_K \mathcal{H}\nabla_X KV = 0, \ \mathcal{T}_V KW = 0, \ Q_K \mathcal{H}\nabla_V KW = 0
$$
for $V,W\in \Gamma(\ker F_*)$ and $X\in \Gamma((\ker F_*)^{\perp})$.
\end{theorem}

\begin{proof}
By (\ref{eq: 2-9}), we have $(\nabla F_*)(X, Y) = 0$ for $X,Y\in \Gamma((\ker F_*)^{\perp})$.

Given $V,W\in \Gamma(\ker F_*)$, $X\in \Gamma((\ker F_*)^{\perp})$, and $R\in \{ I,J,K \}$, by using (\ref{eq: 3-2}) and (\ref{eq: 3-3}),
we obtain
\begin{align*}
(\nabla F_*)(X, V) &= -F_* (\nabla_X V)    \\
      &= F_* (R\nabla_X RV)   \\
      &= F_* (R(\mathcal{A}_X RV + \mathcal{H}\nabla_X RV)) = 0
\end{align*}
$\Leftrightarrow$ $R(\mathcal{A}_X RV + Q_R \mathcal{H}\nabla_X RV) = 0$
$\Leftrightarrow$ $\mathcal{A}_X RV = 0, \ Q_R \mathcal{H}\nabla_X RV = 0$

and
\begin{align*}
(\nabla F_*)(V, W)
      &= -F_* (\nabla_V W)   \\
      &= F_* (R\nabla_V RW)    \\
      &= F_* (R(\mathcal{T}_V RW + \mathcal{H}\nabla_V RW)) = 0
\end{align*}
$\Leftrightarrow$ $R(\mathcal{T}_V RW + Q_R\mathcal{H}\nabla_V RW) = 0$
$\Leftrightarrow$ $\mathcal{T}_V RW = 0, \ Q_R\mathcal{H}\nabla_V RW = 0$.

Hence,
$$
a) \Leftrightarrow b), \quad a) \Leftrightarrow c), \quad a) \Leftrightarrow d).
$$

Therefore, the result follows.
\end{proof}

\begin{lemma}
Let $F$ be a h-Lagrangian submersion from a hyperk\"{a}hler
manifold $(M,I,J,K,g_M)$ onto a Riemannian manifold $(N, g_N)$
such that $(I,J,K)$ is a h-Lagrangian basis. Then the following
conditions are equivalent:

a) the map $F$ is a totally geodesic map.

b) $\mathcal{A}_X IV = 0$ and $\mathcal{T}_V IW = 0$ for $V,W\in \Gamma(\ker F_*)$ and $X\in \Gamma((\ker F_*)^{\perp})$.

c) $\mathcal{A}_X KV = 0$ and $\mathcal{T}_V KW = 0$ for $V,W\in \Gamma(\ker F_*)$ and $X\in \Gamma((\ker F_*)^{\perp})$.

d) $\mathcal{A}_X JV = 0$ and $\mathcal{T}_V JW = 0$ for $V,W\in \Gamma(\ker F_*)$ and $X\in \Gamma((\ker F_*)^{\perp})$.
\end{lemma}

\begin{proof}
By the proof of Theorem \ref{thm: 3-4}, we have $a) \Leftrightarrow b)$ and $a) \Leftrightarrow c)$.

Given $V,W\in \Gamma(\ker F_*)$ and $X\in \Gamma((\ker F_*)^{\perp})$, we get
\begin{align*}
(\nabla F_*)(X, V) &= -F_* (\nabla_X V)    \\
      &= F_* (J\nabla_X JV)   \\
      &= F_* (J(\mathcal{A}_X JV + \mathcal{V}\nabla_X JV)) \\
      &= F_* J\mathcal{A}_X JV = 0
\end{align*}
$\Leftrightarrow$ $\mathcal{A}_X JV = 0$

and
\begin{align*}
(\nabla F_*)(V, W)
      &= -F_* (\nabla_V W)   \\
      &= F_* (J\nabla_V JW)    \\
      &= F_* (J(\mathcal{T}_V JW + \mathcal{V}\nabla_V JW))   \\
      &= F_* J\mathcal{T}_V JW  = 0
\end{align*}
$\Leftrightarrow$ $\mathcal{T}_V JW  = 0$,

which implies $a) \Leftrightarrow d)$.

Therefore, we obtain the result.
\end{proof}

\begin{theorem}\label{thm: 3-5}
Let $F$ be a h-anti-invariant submersion from a hyperk\"{a}hler
manifold $(M,I,J,K,g_M)$ onto a Riemannian manifold $(N, g_N)$
such that $(I,J,K)$ is a h-anti-invariant basis.
Then the following conditions are equivalent:

a) the map $F$ is harmonic.

b) $Q_I(trace (\mathcal{T})) = 0$ on $\ker F_*$ and $trace (I\mathcal{T}_V) = 0$ on $\ker F_*$ for $V\in \Gamma(\ker F_*)$.

c) $Q_J(trace (\mathcal{T})) = 0$ on $\ker F_*$ and $trace (J\mathcal{T}_V) = 0$ on $\ker F_*$ for $V\in \Gamma(\ker F_*)$.

d) $Q_K(trace (\mathcal{T})) = 0$ on $\ker F_*$ and $trace (K\mathcal{T}_V) = 0$ on $\ker F_*$ for $V\in \Gamma(\ker F_*)$.
\end{theorem}

\begin{proof}
By (\ref{eq: 2-9}), we know that the map $F$ is harmonic if and only if $\displaystyle{ \sum_{i=1}^m \mathcal{T}_{e_i} e_i = 0 }$ for any
local orthonormal frame $\{ e_1,e_2,\cdots,e_m \}$ of $\ker F_*$.

Given $V,W\in \Gamma(\ker F_*)$, $R\in \{ I,J,K \}$, and a local orthonormal frame $\{ e_1,e_2,\cdots,e_m \}$ of $\ker F_*$,
by using (\ref{eq: 3-2}) and (\ref{eq: 3-3}),
we obtain
\begin{align*}
\mathcal{T}_V RW &= \mathcal{V}\nabla_V RW    \\
      &= \mathcal{V}R\nabla_V W    \\
      &= \mathcal{V}R(\mathcal{T}_V W + \mathcal{V}\nabla_V W)   \\
      &= \mathcal{V}RP_R\mathcal{T}_V W
\end{align*}
so that by using (\ref{eq: 2-7}) and (\ref{eq: 2-71}), we get
\begin{align*}
g_M(\sum_{i=1}^m \mathcal{T}_{e_i} e_i, RV)
      &= \sum_{i=1}^m g_M(\mathcal{T}_{e_i} e_i, RV)   \\
      &= \sum_{i=1}^m g_M(P_R \mathcal{T}_{e_i} e_i, RV)    \\
      &= -\sum_{i=1}^m g_M(RP_R \mathcal{T}_{e_i} e_i, V)    \\
      &= -\sum_{i=1}^m g_M(\mathcal{V}RP_R \mathcal{T}_{e_i} e_i, V)    \\
      &= -\sum_{i=1}^m g_M(\mathcal{T}_{e_i} Re_i, V)    \\
      &= \sum_{i=1}^m g_M(Re_i, \mathcal{T}_{e_i} V)    \\
      &= \sum_{i=1}^m g_M(Re_i, \mathcal{T}_V {e_i})    \\
      &= -\sum_{i=1}^m g_M(e_i, R\mathcal{T}_V {e_i}) = 0
\end{align*}
$\Leftrightarrow$ $trace (R\mathcal{T}_V) = 0$ for $V\in \Gamma(\ker F_*)$.

Hence,
$$
a) \Leftrightarrow b), \quad a) \Leftrightarrow c), \quad a) \Leftrightarrow d).
$$

Therefore, the result follows.
\end{proof}

\begin{lemma}
Let $F$ be a h-Lagrangian submersion from a hyperk\"{a}hler
manifold $(M,I,J,K,g_M)$ onto a Riemannian manifold $(N, g_N)$
such that $(I,J,K)$ is a h-Lagrangian basis.
Then the map $F$ is harmonic.
\end{lemma}

\begin{proof}
Since $J(\ker F_*) = \ker F_*$, we can choose a local orthonormal frame $\{ e_1,Je_1,$ $\cdots,e_k,Je_k \}$ of $\ker F_*$.

Given $V,W\in \Gamma(\ker F_*)$, we have 
\begin{align*}
\mathcal{T}_V JW &= \mathcal{H}\nabla_V JW    \\
      &= \mathcal{H}J\nabla_V W    \\
      &= \mathcal{H}J(\mathcal{T}_V W + \mathcal{V}\nabla_V W)   \\
      &= J\mathcal{T}_V W
\end{align*}
so that
\begin{align*}
\sum_{i=1}^k (\mathcal{T}_{e_i} e_i + \mathcal{T}_{Je_i} Je_i)
      &= \sum_{i=1}^k (\mathcal{T}_{e_i} e_i + J\mathcal{T}_{Je_i} e_i)   \\
      &= \sum_{i=1}^k (\mathcal{T}_{e_i} e_i + J\mathcal{T}_{e_i} Je_i)   \\
      &= \sum_{i=1}^k (\mathcal{T}_{e_i} e_i + J^2\mathcal{T}_{e_i} e_i)    \\
      &= \sum_{i=1}^k (\mathcal{T}_{e_i} e_i - \mathcal{T}_{e_i} e_i)   \\
      &= 0
\end{align*}

Therefore, the result follows.
\end{proof}

\section{Decomposition theorems}\label{decom}

First of all, we remind some notions.
Let $(M,g)$ be a Riemannian manifold and $L$ a foliation of $M$.
Let $\xi$ be the tangent bundle of $L$ considered as a subbundle of the tangent bundle $TM$ of $M$.

We call $L$ a {\em totally umbilic foliation} \cite{PR} of $M$ if
\begin{equation}\label{eq: 4-1}
h(X, Y) = g(X, Y)H \quad \text{for} \ X,Y\in \Gamma(\xi),
\end{equation}
where $h$ is the second fundamental form of $L$ in $M$ and $H$ is the mean curvature vector field of $L$ in $M$.

The foliation $L$ is said to be a {\em spheric foliation} \cite{PR} if it is a totally umbilic foliation and
\begin{equation}\label{eq: 4-2}
\nabla_X H\in \Gamma(\xi) \quad \text{for} \ X\in \Gamma(\xi),
\end{equation}
where $\nabla$ is the Levi-Civita connection of $g$.

We call $L$ a {\em totally geodesic foliation} \cite{PR} of $M$ if
\begin{equation}\label{eq: 4-3}
\nabla_X Y\in \Gamma(\xi) \quad \text{for} \ X,Y\in \Gamma(\xi).
\end{equation}
Let $(M_1,g_1)$ and $(M_2,g_2)$ be Riemannian manifolds, $f_i : M_1 \times M_2 \mapsto \mathbb{R}$ a positive $C^{\infty}$-function,
and $\pi_i : M_1 \times M_2 \mapsto M_i$ the canonical projection for $i = 1,2$.

We call $M_1 \times_{(f_1,f_2)} M_2$ a {\em double-twisted product manifold} \cite{PR} of $(M_1,g_1)$ and $(M_2,g_2)$ if it is the product manifold
$M := M_1 \times M_2$ with the Riemannian metric $g$ such that
\begin{equation}\label{eq: 4-4}
g(X, Y) = f_1^2\cdot g_1({\pi_1}_* X, {\pi_1}_* Y) + f_2^2\cdot g_2({\pi_2}_* X, {\pi_2}_* Y) \quad \text{for} \ X,Y\in \Gamma(TM).
\end{equation}
We call $M_1 \times_{(f_1,f_2)} M_2$ {\em non-trivial} if both $f_1$ and $f_2$ are not constant functions.

The Riemannian manifold $M_1 \times_f M_2$ is said to be a {\em twisted product manifold} \cite{PR} of $(M_1,g_1)$ and $(M_2,g_2)$
if $M_1 \times_f M_2 = M_1 \times_{(1,f)} M_2$.

We call $M_1 \times_f M_2$ {\em non-trivial} if $f$ is not a constant function.

The twisted product manifold $M_1 \times_f M_2$ is said to be a {\em warped product manifold} \cite{PR} of $(M_1,g_1)$ and $(M_2,g_2)$
if $f$ depends only on the points of $M_1$. (i.e., $f\in C^{\infty}(M_1, \mathbb{R})$)

Let $M_1$ and $M_2$ be connected $C^{\infty}$-manifolds and $M$ the product manifold $M_1 \times M_2$.
Let $\pi_i : M \mapsto M_i$ be the canonical projection for $i=1,2$.
Let $\xi_i := \ker {\pi_{3-i}}_*$ and $P_i : TM \mapsto \xi_i$ the vector bundle projection such that $TM = \xi_1 \oplus \xi_2$.
And let $L_i$ be the canonical foliation of $M$ by the integral manifolds of $\xi_i$ for $i = 1,2$.

\begin{proposition}\label{prop: 4-1} \cite{PR}
Let $g$ be a Riemannian metric on the product manifold $M_1 \times M_2$ and assume that the canonical foliations $L_1$ and $L_2$
 intersect perpendicularly everywhere. Then $g$ is the metric of

a) a double-twisted product manifold $M_1 \times_{(f_1,f_2)} M_2$ if and only if $L_1$ and $L_2$ are totally umbilic foliations,

b) a twisted product manifold $M_1 \times_f M_2$ if and only if $L_1$ is a totally geodesic foliation and $L_2$ is a totally umbilic foliation,

c) a warped product manifold $M_1 \times_f M_2$ if and only if $L_1$ is a totally geodesic foliation and $L_2$ is a spheric foliation,

d) a (usual) Riemannian product manifold $M_1 \times M_2$ if and only if $L_1$ and $L_2$ are totally geodesic foliations.
\end{proposition}

Let $F$ be a Riemannian submersion from a Riemannian manifold $(M,g_M)$ onto a Riemannian manifold $(N,g_N)$ such that the distributions
$\ker F_*$ and $(\ker F_*)^{\perp}$ are integrable. Then we denote by $M_{\ker F_*}$ and $M_{(\ker F_*)^{\perp}}$ the integral manifolds of
$\ker F_*$ and $(\ker F_*)^{\perp}$, respectively. We also denote by $H$ and $H^{\perp}$ the mean curvature vector fields of
$\ker F_*$ and $(\ker F_*)^{\perp}$, respectively. i.e., $\displaystyle{ H = \frac{1}{m}\sum_{i=1}^m \mathcal{T}_{e_i} e_i }$ and
$\displaystyle{ H^{\perp} = \frac{1}{n}\sum_{i=1}^n \mathcal{A}_{v_i} v_i }$ for a local orthonormal frame $\{ e_1,\cdots,e_m \}$ of
$\ker F_*$ and a local orthonormal frame $\{ v_1,\cdots,v_n \}$ of  $(\ker F_*)^{\perp}$.

Using Proposition \ref{prop: 4-1}, Theorem \ref{thm: 3-2}, and Theorem \ref{thm: 3-3}, we get

\begin{theorem}
Let $F$ be a h-anti-invariant submersion from a hyperk\"{a}hler
manifold $(M,I,J,K,g_M)$ onto a Riemannian manifold $(N, g_N)$
such that $(I,J,K)$ is a h-anti-invariant basis.
Then the following conditions are equivalent:

a) $(M,g_M)$ is locally a Riemannian product manifold of the form $M_{(\ker F_*)^{\perp}} \times M_{\ker F_*}$.

b)
$$
g_M(\mathcal{A}_X B_I Y, IV) = g_M(C_I Y, I\mathcal{A}_X V) \ \text{and} \ \mathcal{T}_V B_I X + \mathcal{A}_{C_I X} V\in \Gamma(\mu_I)
$$
for $V\in \Gamma(\ker F_*)$ and $X,Y\in \Gamma((\ker F_*)^{\perp})$.

c)
$$
g_M(\mathcal{A}_X B_J Y, JV) = g_M(C_J Y, J\mathcal{A}_X V) \ \text{and} \ \mathcal{T}_V B_J X + \mathcal{A}_{C_J X} V\in \Gamma(\mu_J)
$$
for $V\in \Gamma(\ker F_*)$ and $X,Y\in \Gamma((\ker F_*)^{\perp})$.

d)
$$
g_M(\mathcal{A}_X B_K Y, KV) = g_M(C_K Y, K\mathcal{A}_X V) \ \text{and} \ \mathcal{T}_V B_K X + \mathcal{A}_{C_K X} V\in \Gamma(\mu_K)
$$
for $V\in \Gamma(\ker F_*)$ and $X,Y\in \Gamma((\ker F_*)^{\perp})$.
\end{theorem}

Using Proposition \ref{prop: 4-1}, Lemma \ref{lem: 3-2}, and Lemma \ref{lem: 3-3}, we obtain

\begin{lemma}
Let $F$ be a h-Lagrangian submersion from a hyperk\"{a}hler
manifold $(M,I,J,K,g_M)$ onto a Riemannian manifold $(N, g_N)$
such that $(I,J,K)$ is a h-Lagrangian basis. Then the following
conditions are equivalent:

a) $(M,g_M)$ is locally a Riemannian product manifold of the form $M_{(\ker F_*)^{\perp}} \times M_{\ker F_*}$.

b)
$$
\mathcal{A}_X IY = 0 \ \text{and} \ \mathcal{T}_V IX = 0
$$
for $X,Y\in \Gamma((\ker F_*)^{\perp})$ and  $V\in \Gamma(\ker F_*)$.

c)
$$
\mathcal{A}_X KY = 0 \ \text{and} \ \mathcal{T}_V KX = 0
$$
for $X,Y\in \Gamma((\ker F_*)^{\perp})$ and  $V\in \Gamma(\ker F_*)$.

d)
$$
\mathcal{A}_X JY = 0 \ \text{and} \ \mathcal{T}_V JX = 0
$$
for $X,Y\in \Gamma((\ker F_*)^{\perp})$ and  $V\in \Gamma(\ker F_*)$.
\end{lemma}

Now, we deal with the geometry of distributions $\ker F_*$ and $(\ker F_*)^{\perp}$.

\begin{theorem}\label{thm: 4-4}
Let $F$ be a Riemannian submersion from a Riemannian manifold $(M,g_M)$ onto a Riemannian manifold $(N,g_N)$.
Assume that the distribution $(\ker F_*)^{\perp}$ defines a totally umbilic foliation on $M$.
Then the distribution $(\ker F_*)^{\perp}$ also defines a totally geodesic foliation on $M$.
\end{theorem}

\begin{proof}
Given $X,Y\in \Gamma((\ker F_*)^{\perp})$ and  $V\in \Gamma(\ker F_*)$, we get
\begin{equation}\label{eq: 4-5}
g_M(\nabla_X Y, V) = g_M(\mathcal{A}_X Y, V) = g_M(X, Y) g_M(H^{\perp}, V)
\end{equation}
and
\begin{equation}\label{eq: 4-6}
g_M(\nabla_X Y, V) = -g_M(Y, \nabla_X V) = -g_M(Y, \mathcal{A}_X V).
\end{equation}
Comparing (\ref{eq: 4-5}) and (\ref{eq: 4-6}), we obtain $\mathcal{A}_X V = -g_M(H^{\perp}, V)X$.

Hence,
\begin{equation}\label{eq: 4-7}
g_M(\mathcal{A}_X V, X) = -g_M(H^{\perp}, V)||X||^2.
\end{equation}
But
\begin{align*}
g_M(\mathcal{A}_X V, X) &= g_M(\nabla_X V, X)    \\
      &= -g_M(V, \nabla_X X)    \\
      &= -g_M(V, \mathcal{A}_X X)   \\
      &= 0 \ (\text{by} \ (\ref{eq: 2-6}))
\end{align*}
so that from (\ref{eq: 4-7}), we have $H^{\perp} = 0$.

Therefore, the result follows.
\end{proof}

\begin{remark}
From the equation: $\mathcal{A}_X Y = - \mathcal{A}_Y X$ for $X,Y\in \Gamma((\ker F_*)^{\perp})$, we can obtain Theorem \ref{thm: 4-4}.
But from the equation: $\mathcal{T}_V W =  \mathcal{T}_W V$ for $V,W\in \Gamma(\ker F_*)$, there are no theorems like Theorem \ref{thm: 4-4}
on $\ker F_*$.
\end{remark}

\begin{theorem}\label{thm: 4-6}
Let $F$ be a h-anti-invariant submersion from a hyperk\"{a}hler
manifold $(M,I,J,K,g_M)$ onto a Riemannian manifold $(N, g_N)$
such that $(I,J,K)$ is a h-anti-invariant basis.
Then the following conditions are equivalent:

a) the distribution $\ker F_*$ defines a totally umbilic foliation on $M$.

b)
$$
\mathcal{T}_V B_I X + \mathcal{H}\nabla_V C_I X = -g_M(H, X)IV
$$
for $V\in \Gamma(\ker F_*)$ and $X\in \Gamma((\ker F_*)^{\perp})$.

c)
$$
\mathcal{T}_V B_J X + \mathcal{H}\nabla_V C_J X = -g_M(H, X)JV
$$
for $V\in \Gamma(\ker F_*)$ and $X\in \Gamma((\ker F_*)^{\perp})$.

d)
$$
\mathcal{T}_V B_K X + \mathcal{H}\nabla_V C_K X = -g_M(H, X)KV
$$
for $V\in \Gamma(\ker F_*)$ and $X\in \Gamma((\ker F_*)^{\perp})$.
\end{theorem}

\begin{proof}
Given $V,W\in \Gamma(\ker F_*)$, $X\in \Gamma((\ker F_*)^{\perp})$, and $R\in \{ I,J,K \}$, we obtain
\begin{align*}
g_M(\mathcal{T}_V W, X) &= g_M(\nabla_V RW, RX)    \\
      &= -g_M(RW, \nabla_V B_R X + \nabla_V C_R X)   \\
      &= -g_M(RW, \mathcal{T}_V B_R X + \mathcal{H}\nabla_V C_R X)
\end{align*}
so that it is easy to check that
$$
\mathcal{T}_V W = g_M(V, W)H \Leftrightarrow \mathcal{T}_V B_R X + \mathcal{H}\nabla_V C_R X = -g_M(H, X)RV. 
$$
Hence,
$$
a) \Leftrightarrow b), \quad a) \Leftrightarrow c), \quad a) \Leftrightarrow d).
$$

Therefore, we get the result.
\end{proof}

\begin{lemma}\label{lem: 4-6}
Let $F$ be a h-Lagrangian submersion from a hyperk\"{a}hler
manifold $(M,I,J,K,g_M)$ onto a Riemannian manifold $(N, g_N)$
such that $(I,J,K)$ is a h-Lagrangian basis. Then the following
conditions are equivalent:

a) the distribution $\ker F_*$ defines a totally umbilic foliation on $M$.

b) $\mathcal{T}_V IX = -g_M(H, X)IV$ for $X\in \Gamma((\ker F_*)^{\perp})$ and  $V\in \Gamma(\ker F_*)$.

c) $\mathcal{T}_V KX = -g_M(H, X)KV$ for $X\in \Gamma((\ker F_*)^{\perp})$ and  $V\in \Gamma(\ker F_*)$.

d) $\mathcal{T}_V JX = -g_M(H, X)JV$ for $X\in \Gamma((\ker F_*)^{\perp})$ and  $V\in \Gamma(\ker F_*)$.
\end{lemma}

\begin{proof}
By the proof of Theorem \ref{thm: 4-6}, we have $a) \Leftrightarrow b)$ and $a) \Leftrightarrow c)$.

Given $V,W\in \Gamma(\ker F_*)$ and $X\in \Gamma((\ker F_*)^{\perp})$, we get
\begin{align*}
g_M(\mathcal{T}_V W, X) &= g_M(\nabla_V JW, JX)    \\
      &= -g_M(JW, \nabla_V JX)   \\
      &= -g_M(JW, \mathcal{T}_V JX)
\end{align*}
so that we easily check that
$$
\mathcal{T}_V W = g_M(V, W)H \Leftrightarrow \mathcal{T}_V JX = -g_M(H, X)JV.
$$
Hence, $a) \Leftrightarrow d)$.

Therefore, the result follows.
\end{proof}

Using Proposition \ref{prop: 4-1}, Theorem \ref{thm: 3-2}, and Theorem \ref{thm: 4-6}, we get

\begin{theorem}
Let $F$ be a h-anti-invariant submersion from a hyperk\"{a}hler
manifold $(M,I,J,K,g_M)$ onto a Riemannian manifold $(N, g_N)$
such that $(I,J,K)$ is a h-anti-invariant basis.
Then the following conditions are equivalent:

a) $(M,g_M)$ is locally a twisted product manifold of the form $M_{(\ker F_*)^{\perp}} \times M_{\ker F_*}$.

b)
$$
g_M(\mathcal{A}_X B_I Y, IV) = g_M(C_I Y, I\mathcal{A}_X V) \ \text{and} \ \mathcal{T}_V B_I X + \mathcal{H}\nabla_V C_I X = -g_M(H, X)IV
$$
for $V\in \Gamma(\ker F_*)$ and $X,Y\in \Gamma((\ker F_*)^{\perp})$.

c)
$$
g_M(\mathcal{A}_X B_J Y, JV) = g_M(C_J Y, J\mathcal{A}_X V) \ \text{and} \ \mathcal{T}_V B_J X + \mathcal{H}\nabla_V C_J X = -g_M(H, X)JV
$$
for $V\in \Gamma(\ker F_*)$ and $X,Y\in \Gamma((\ker F_*)^{\perp})$.

d)
$$
g_M(\mathcal{A}_X B_K Y, KV) = g_M(C_K Y, K\mathcal{A}_X V) \ \text{and} \ \mathcal{T}_V B_K X + \mathcal{H}\nabla_V C_K X = -g_M(H, X)KV
$$
for $V\in \Gamma(\ker F_*)$ and $X,Y\in \Gamma((\ker F_*)^{\perp})$.
\end{theorem}

Using Proposition \ref{prop: 4-1}, Lemma \ref{lem: 3-2}, and Lemma \ref{lem: 4-6}, we have

\begin{lemma}
Let $F$ be a h-Lagrangian submersion from a hyperk\"{a}hler
manifold $(M,I,J,K,g_M)$ onto a Riemannian manifold $(N, g_N)$
such that $(I,J,K)$ is a h-Lagrangian basis. Then the following
conditions are equivalent:

a) $(M,g_M)$ is locally a twisted product manifold of the form $M_{(\ker F_*)^{\perp}} \times M_{\ker F_*}$.

b)
$$
\mathcal{A}_X IY = 0 \ \text{and} \ \mathcal{T}_V IX = -g_M(H, X)IV
$$
for $X,Y\in \Gamma((\ker F_*)^{\perp})$ and  $V\in \Gamma(\ker F_*)$.

c)
$$
\mathcal{A}_X KY = 0 \ \text{and} \ \mathcal{T}_V KX = -g_M(H, X)KV
$$
for $X,Y\in \Gamma((\ker F_*)^{\perp})$ and  $V\in \Gamma(\ker F_*)$.

d)
$$
\mathcal{A}_X JY = 0 \ \text{and} \ \mathcal{T}_V JX = -g_M(H, X)JV
$$
for $X,Y\in \Gamma((\ker F_*)^{\perp})$ and  $V\in \Gamma(\ker F_*)$.
\end{lemma}

Now, we consider the non-existence of some types of Riemannian submersions.

Using Proposition \ref{prop: 4-1} and Theorem \ref{thm: 4-4}, we get

\begin{theorem}
Let $(M,E,g_M)$ be an almost quaternionic Hermitian manifold and $(N, g_N)$ a Riemannian manifold.
Then there does not exist a h-anti-invariant submersion from $M = (M,E,g_M)$ onto $(N, g_N)$ such that
$M$ is locally a non-trivial double-twisted product manifold of the form $M_{(\ker F_*)^{\perp}} \times M_{\ker F_*}$.
\end{theorem}

\begin{lemma}
Let $(M,E,g_M)$ be an almost quaternionic Hermitian manifold and $(N, g_N)$ a Riemannian manifold.
Then there does not exist a h-Lagrangian submersion from $M = (M,E,g_M)$ onto $(N, g_N)$ such that
$M$ is locally a non-trivial double-twisted product manifold of the form $M_{(\ker F_*)^{\perp}} \times M_{\ker F_*}$.
\end{lemma}

\begin{theorem}
Let $(M,E,g_M)$ be an almost quaternionic Hermitian manifold and $(N, g_N)$ a Riemannian manifold.
Then there does not exist a h-anti-invariant submersion from $M = (M,E,g_M)$ onto $(N, g_N)$ such that
$M$ is locally a non-trivial twisted product manifold of the form $M_{\ker F_*} \times M_{(\ker F_*)^{\perp}}$.
\end{theorem}

\begin{lemma}
Let $(M,E,g_M)$ be an almost quaternionic Hermitian manifold and $(N, g_N)$ a Riemannian manifold.
Then there does not exist a h-Lagrangian submersion from $M = (M,E,g_M)$ onto $(N, g_N)$ such that
$M$ is locally a non-trivial twisted product manifold of the form $M_{\ker F_*} \times M_{(\ker F_*)^{\perp}}$.
\end{lemma}

\end{document}